\documentclass[12pt]{article}
\usepackage[a4paper, left=2cm, right=2cm, bottom=3cm, top=3cm]{geometry}
\usepackage{amsmath}
\usepackage{srcltx}
\usepackage{amsfonts}
\usepackage{amssymb}
\usepackage{amsthm}
\usepackage{lipsum}
\usepackage{psfrag}
\usepackage{verbatim}
\usepackage{graphicx} 
\usepackage{color}
\usepackage{faktor}
\providecommand{\U}[1]{\protect\rule{.1in}{.1in}}

\newtheorem{theorem}{Theorem}[section]

\newtheorem{definition}[theorem]{Definition}

\newtheorem{proposition}[theorem]{Proposition}
\newtheorem{remark}[theorem]{Remark}

\newcommand{\R}{\ensuremath{\mathbb{R}}}

\begin{document}
\title{Observability of linear systems on the Heisenberg Lie group}
\author{Victor Ayala\\ Instituto de Alta Investigación, Universidad de Tarapac\'a \\ Arica, Chile \\vayala@academicos.uta.cl
\and Thiago Matheus Cavalheiro\\Departamento de Matem\'{a}tica, Universidade Estadual do Norte do Paraná \\ Jacarezinho, Brazil \\thiago\_cavalheiro@hotmail.com 
\and Alexandre J. Santana\\Departamento de Matem\'{a}tica, Universidade Estadual de Maring\'{a}\\ Maring\'a, Brazil \\ajsantana@uem.br\\ 
}
\maketitle



\begin{abstract}
{In control theory, understanding the observability property of a system is crucial for effectively managing and controlling dynamical systems. This property empowers us to deduce the internal state of a system from its outputs over time, even when direct measurements are impossible. By harnessing observability, we can accurately estimate the complete state of a system and reconstruct its dynamics using just limited information. In this work, we will find conditions for observability of linear systems in the three dimensional Heisenberg group $\mathcal{H}$. Considering the homomorphisms between the group and its simply connected subgroups, whose kernel is denoted by $K$, we will find sufficient conditions for observability on the system using a quotient space $\mathcal{H}/K$ as the output.}
\end{abstract}



%

	

\maketitle
\section{Introduction}

	In the context of control theory, the observability property of a system relates to the ability to infer the internal state of a dynamical system based on its outputs over time. This property is critical for designing systems in which the full state space cannot be directly measured but must be estimated through available outputs. Moreover,  observability enables us to reconstruct the entire dynamics from limited information. 
 
 A compelling example comes from the renowned Pontryaguin's Maximum Principle, \cite{Pontryaguin} which earned him the Lenin Prize in Russia in the early sixties.
 
 Consider the challenge of stopping a train at a railway station as quickly as possible. 
 
 This situation represents a linear control system on the Euclidean plane. In this scenario, the first variable, $x(t)$, indicates the distance from the train to the station $(0,0)$ at time $t$, while the second one, its derivative, $y(t)$, represents the velocity. The control mechanism is based on acceleration. Therefore, the optimal problem reduces to taking any initial condition on the plane to the origin in minimum time.

Remarkably, we can fully reconstruct an optimal trajectories on the plane by using the projection onto the x-axis as the output, from a curve of distances. This means that the information provided by a one-dimensional subspace allows us to recover the entire optimal curve in the plane for any initial position of the train. In contrast, the velocity alone does not provide sufficient information. 

This example underscores the importance of observability: it enables us to develop effective solutions with fewer variables while highlighting the critical components of the system. Embracing this concept enhances efficiency and sharpens our focus on what truly matters in control systems.

The dynamic of the train example is determined by a real matrix of order two and a column vector. This model can be readily extended to n-dimensional Euclidean space as follows.

Consider the linear control system 
	\begin{equation*}
		\Sigma: \left\{\begin{array}{ccl}
			\dot{x(t)} &=& Ax(t) + Bu(t)\\
			y(x(t)) &=& Cx(t) 
		\end{array}\right., 
	\end{equation*}
	where $A \in \R^{n\times n}$, $B \in \R^{n \times m}$, $C \in \R^{l \times n}$, with $l < n$ and $u \in L_{const}(\R,\R^n)$. 

 The variable $y$ is the output map defined on $\R^{l}$ and computed by the matrix $C$. Let $\varphi_t(x,u)$ be the solution of the system starting in $x$ at the time $t \in \R$, and control $u$.  
 
 The system $\Sigma$ is \textit{observable in} $x$ if for every $z \in \R^n\setminus \{x\}$ there is a $t >0$ and a control $u$ such that $C\varphi_t(x,u) \neq C\varphi_t(z,u)$. Furthermore, $\Sigma$ is global observable (observable), if it is observable in any element $x$ of the $n$-dimensional vector space. So, every initial condition can be distinguished from any other state through a specific solution to the system.

 There exists a computable algebraic characterization of this property, as follows,
$\Sigma$ is observable if and only if  $\hbox{rank}(\mathcal{O}) = n$.
   Here,  $\mathcal{O} = [C \hbox{ }CA\hbox{ }...\hbox{ }CA^{n-1}]^T$, 
 \cite[Theorem 23]{Sontag}

In 1999, it was introduced the following notion: \cite{AyTi} 

A linear control system on a connected Lie group $G,$ is determined by the family,

\begin{equation*}\Sigma _{G} :\text{}\overset{ \cdot }{g (t)}\text{} =\mathcal{X} (g (t)) +\sum _{j =1}^{m}u_{j} (t) Y^{j} (g (t)) ,\text{}g (t) \in G ,\text{}t \in \mathbb{R}\text{} ,u \in \mathcal{U} ,
\end{equation*}of ordinary differential equations parametrized by the class  
 $\mathcal{U} =L_{l o c}^{1} (\Omega )\text{,}$ of admissible controls, i.e., by locally integrable functions $u :[0 ,T_{u}] \rightarrow \Omega  \subset $$\mathbb{R}^{m}$. The set $\Omega $ is closed and $0 \in i n t (\Omega )$. 

The drift $\mathcal{X}$ is a linear vector field, which means, exactly as in the matrix case, that its flows is a 1-parameter group of automorphisms of $G$.  And, for any $j\text{,}$ the control vector $Y^{j} \in \mathfrak{g}$, considered as left-invariant vector field. If $\Omega  =\mathbb{R}^{m}\text{,}$ the system is called unrestricted. Otherwise, $\Sigma _{G}$\ is restricted,  \cite{AyTi}.

\begin{remark} The observability property of $\Sigma$ doesn´t depends of the columns of the matrix $B$. The same is true for $\Sigma _{G}$.  It is shown that the invariant vector fields in the system do not influence observability.
\end{remark}

Let $K \subset G$ a closed subgroup. By using the projection  $\pi_K: G \longrightarrow G/K$, onto the homogeneous space, 
 the authors extend the observability results from Euclidean spaces to a Lie group $G$, \cite{AyalaHaci}. 

Our primary objective is to apply these results to arbitrary linear control systems on the connected and simply connected nilpotent Heisenberg Lie group of dimension three. To accomplish this, we provide a detailed description of the linear and invariant vector fields on \(\mathcal{H}\) in order to construct potential linear control systems. Additionally, we identify the closed subgroups and their algebras, along with the projection map from the group to the corresponding homogeneous space, which play the role as the output map. With this information, we characterize observability in all these cases based on the results mentioned.

 The concept of observability is not new and there are many papers studying this property to different class of systems, (see \cite{Kalman1}, \cite{Kalman2} and \cite{Sontag}, for instance). Nowadays, its very used in dynamical and algebraic systems (see \cite{Berger}, \cite{Etlili} and \cite{Sivalingam} for instance). 
 
 Our most important tools will be the ideas and results introduced in \cite{AyalaHaci} and \cite{Ayalaetal}. 
	
	This article is divided as follows: the Section 2 is dedicated to introduce the basic properties and definitions and Section 3 contains the main results, including the conditions the system has to satisfy to ensure global observability. Also in section 3, we briefly discussed the conditions when the system is not observable. 
	
	\section{Preliminaries and Notations}
	
	As we saw in the introduction, the observability property of a linear control system $\Sigma$ on the Euclidean space  its determined by the pair 
$(A,y)$, i.e., the drift and the output map.
The flow of $A$ in $\Sigma$ reads as $e^{tA}$, which is an automorphism. In the Lie group context this matrix is replaced by a linear vector fields, which means by definition that its flows is a 1-parameter groups of Aut(G), the Lie group of G-automorphisms, exactly as in the Euclidean case. In a more algebraic way, it is worth to mention that the Lie bracket between the drift and any left invariant vector field gives a new element of the Lie algebra, in both cases.
On the other hand, the linear map  $C$ is changed by a projection from the group to a homogeneous space of $G$. 
	
 Denotes by $\mathfrak {X}(G)$ the Lie algebra of all smooth vector fields on \( G \), then   the normalizer of \( \mathfrak{g} \) related to $\mathfrak{X}(G)$ appears in \cite{AyTi}, and is defined by 
	
	\[
	\text{norm}_{\mathfrak{X}(G)}(\mathfrak{g}) = \{ \mathfrak {X}{} \in \mathfrak{X}(G) \mid \text{ad}(\mathcal{X})(Y) = [\mathcal{X},Y] \in \mathfrak{g}, \ \forall Y \in \mathfrak{g} \}.
	\]
	
	A vector field $\mathcal{X} \in \text{norm}_{\mathfrak{X}(G)}(\mathfrak{g})$ is called affine. And, if $\mathcal{X}(e)=0$ then it is linear, where $e$ stands for the identity element of $G$. It is important to note that for any constant admissible control in either the Euclidean or Lie group context, all vector fields in these systems belong to the normalizer.
	
	Consider $K$ a subgroup of $G$ and the canonical projection $\pi_K: G \longrightarrow G/K$ as output map. Hence, we are willing to introduce the following notion.
		
	\begin{definition} The pair $(\mathcal{X}, \pi_K)$ is called a general pair on $G$.
	\end{definition}
	
	Now, denote by $\varphi_t$ the flow of the vector field $\mathcal{X}$. Hence the main concepts of this work are presented.

	\begin{definition}A general pair $(\mathcal{X},\pi_K)$ is said to be:
		\begin{itemize}
			\item[1-] observable at $x_1$ if for all $x_2 \in G\setminus\{x_1\}$, there exists a $t \geq 0$ such that $\pi_K(\varphi_t(x_1)) \neq \pi_K(\varphi_t(x_2))$. 
			\item[2-] locally observable at $x_1$ if there exists a neighborhood of $x_1$ such that the condition $1$ is satisfied for each $x$ in the neighborhood.
			\item[3-] observable (locally observable) if it is observable (locally observable) for every $x \in G$. 
		\end{itemize}
	\end{definition}
	
	According to \cite[Corollary 2.6]{Ayalaetal}, the linear pair $(\mathcal{X},\pi_K)$ is \textbf{locally observable} if, and only if, the set 
	\begin{equation*}
		I = \{(x,y,z) \in H: \varphi_t(x,y,z) \in K, \forall t \in \R\},
	\end{equation*}
	is discrete. Moreover, $(\mathcal{X},\pi_K)$ is \textbf{observable} if, and only if, the pair is locally observable and 
 $$Fix(\varphi) \cap K = \{e\},$$ where $Fix(\varphi)$ is the set of fixed points of the flow of $\varphi$, (see \cite[Theorem 2.5]{AyalaHaci}).
	
\begin{remark}\label{remarkimp} Consider $G$ a connected Lie group and $H_1,H_2 \subset G$ subgroups of $G$. For $i = 1,2$, assume $h_i: G \longrightarrow H_i$, are homomorphisms whose corresponding kernels are $K_1$ and $K_2$. If $K_1 = K_2$, then $(\mathcal{X}, \pi_{K_1})$ is observable if, and only if, $(\mathcal{X}, \pi_{K_2})$ is observable. 
 
To prove this statement considering $I_i = \{p \in G: \varphi_t(p) \in K_i, \forall t \in \R\}$ for $i=1,2$. It turns out that $I_1 = I_2$, and also   $$Fix(\varphi) \cap K_1 = Fix(\varphi) \cap K_2.$$
  
Finally, we note that the definition of an observable linear pair can be viewed as observable in $K$. Furthermore, most of the subgroups of $\mathcal{H}$ are represented by the kernels of homomorphisms.
	
	\section{The 3-dimensional Heisenberg Lie group} 

In this main section, we begin by reviewing key facts about the Heisenberg Lie group $\mathcal{H}$. Then, after studying the fixed points of the linear vector field, describing the closed and simply connected subgroups $H$, finding the homomorphisms $h: \mathcal{H} \rightarrow H$ and their kernels, we study the observability. 
 	
	\subsection{Linear systems on $\mathcal{H}$}
	
	Consider the Heisenberg group 
	\begin{equation*}
		\mathcal{H} = \left\{
		\begin{bmatrix}
			1 & y & z \\
			0 & 1 & x\\
			0 & 0 & 1
		\end{bmatrix} \in GL_3(\R): x,y,z \in \R\right\}
	\end{equation*}
	endowed with the matrix product. Then $\mathcal{H}$ is a $3-$dimensional Lie group, whose Lie algebra is given by 
	\begin{equation*}
		\mathfrak{h} = \left\{
		\begin{bmatrix}
			0 & y & z \\
			0 & 0 & x\\
			0 & 0 & 0
		\end{bmatrix} \in \mathfrak{gl}_3(\R): x,y,z \in \R\right\},
	\end{equation*}
	with the usual matrix Lie bracket. In particular, $\mathcal{H}$ is diffeomorphic to $\R^3$ with  the product 
	\begin{equation}\label{productH}
		(x,y,z) \cdot (w,s,t) = (x + w + yt, y + s,z + t)
	\end{equation}
	and the Lie algebra being $\R^3$ itself,  with the Lie bracket 
	\begin{equation*}
		[(x,y,z), (a,b,c)] = (0,0,ya-bx). 
	\end{equation*}
	
	According to Jouan \cite{DathAndJouan}, the derivations of $\mathfrak{h}$, in the basis $(X,Y,Z)$ are given by 
	\begin{equation*}
		D = 
		\begin{bmatrix}
			a & b & 0 \\
			c & d & 0 \\
			e & f & a+d
		\end{bmatrix}.
	\end{equation*}
	
	Also, the linear vector fields are expressed in the form 
	\begin{equation*}
		\mathcal{X}(x,y,z) = \left(ax + by, cx + dy, ex + fy + (a+d)z\right)
	\end{equation*}
	which generates the following system
	\begin{equation}\label{odeX}
		\left\{
		\begin{array}{ccl}
			\dot{x} &=& ax + by\\
			\dot{y} &=&  cx + dy\\
			\dot{z} &=& ex + fy + (a+d)z
		\end{array}\right.
	\end{equation}
	whose solution with initial condition $(x,y,z)$ is given by 
	\begin{equation}\label{generalsolutionheis}
		\varphi(t,(x,y,z)) = \left(e^{tA}(x,y), \left\langle e^{t (a+d)}\int_0^t e^{s(A-(a+d)I_2)^T}(e,f)ds,(x,y)\right\rangle + e^{t(a+d)}z\right) 
	\end{equation}
	where $A =\begin{bmatrix}
		a & b \\
		c & d
	\end{bmatrix}$. 
	
	\begin{remark}It is clear that $e^{tA}(x,y)$ is the solution of the first and the second lines in the linear system (\ref{odeX}). Now, for the third line, derivating the map
		\begin{equation*}
			f(t) = \left\langle e^{t (a+d)}\int_0^t e^{s(A-(a+d)I_2)^T}(e,f)ds,(x,y)\right\rangle + e^{t(a+d)}z,
		\end{equation*}
		we get 
		\[
			f'(t) =
			\langle e^{t(a+d)}e^{t(A-(a+d)I_2)^T}(e,f),(x,y) \rangle + (a+d) f(t).
		\]
		
		If we suppose that 
		\begin{equation*}
			e^{tA} = \begin{bmatrix}
				g_{11}(t) & g_{12}(t)\\
				g_{21}(t) & g_{22}(t)
			\end{bmatrix}, 
		\end{equation*}
		we get $e^{t(a+d)}e^{t(A-(a+d)I_2)^T} = (e^{tA})^T$ and 
		\begin{equation*}
			(e^{tA})^T(e,f) = (g_{11}(t)e + g_{21}(t) f,g_{12}(t)e + g_{22}(t) f).  
		\end{equation*}
		
		Therefore 
		\begin{equation*}
			\langle (e^{tA})^T(e,f),(x,y)\rangle  = e (g_{11}(t)x + g_{12}(t)y) + f(g_{21}(t)x + g_{22}(t)y) = e x(t) + fy(t), 
		\end{equation*}
		as claimed. 
	\end{remark}
	
	In particular, considering $\alpha$ and $\beta$ the roots of the characteristic polynomial of $A$, given the matrix $A$, it follows by the Sylvester's formula that 
	\begin{equation}\label{Sylvexp}
		e^{tA} = s_0(t)I + s_1(t)A, 
	\end{equation}
	where
	\begin{equation}\label{s_0}
		s_0(t) = \left\{
		\begin{array}{cc}
			\dfrac{\beta e^{t\alpha} - \alpha e^{t\beta}}{\beta-\alpha},& \alpha \neq \beta  \\
			(1- \alpha t)e^{t\alpha},& \alpha = \beta 
		\end{array}\right.
	\end{equation}
	and 
	\begin{equation}\label{s_1}
		s_1(t) = \left\{
		\begin{array}{cc}
			\dfrac{e^{t\alpha} - e^{t\beta}}{\alpha - \beta},& \alpha \neq \beta  \\
			te^{t\alpha},& \alpha = \beta 
		\end{array}\right.
	\end{equation}
	
	\subsection{Fixed points of $\mathcal{X}$}\label{fixedpoints}
	
As we saw in the previous section, to study global observability, we need to know the fixed points of $\mathcal{X}$, that is, the points  $(x,y,z) \in \mathcal{H}$ such that $(\dot{x}, \dot{y}, \dot{z}) = (0,0,0)$. In the system (\ref{odeX}), we get 
	\begin{equation}\label{linearsystem0}
		\left\{\begin{array}{ccl}
			ax + by &=& 0 \\
			cx + dy &=& 0 \\
			ex + fy + (a+d)z &=& 0 
		\end{array}\right.
	\end{equation}
	
	Considering the matrix $A = \begin{bmatrix}
		a & b \\
		c & d
	\end{bmatrix}$, let us split in cases. 
	
	\noindent\textbf{Case 1: $A$ is invertible:} From the system (\ref{linearsystem0}), we get $x = y = 0$. Then $(a + d)z = 0$, which implies $a = -d$ or $z = 0$. If $a = -d$, then $z$ is an arbitrary element of $\R$. Therefore, we get  
	\begin{equation*}
		\varphi_t(x,y,z) = (0,0,z), \forall t \in \R,
	\end{equation*}
	
	Now, if $a\neq -d$, then $z(t) = e^{(a+d)t}z$. Consequently 
	\begin{equation*}
		\varphi_t(x,y,z) = (0,0,z), \forall t \in \R \iff z = 0. 
	\end{equation*}
	
	Summing up, we get 
	\begin{equation*}
		Fix(\varphi)= \{(x,y,z) \in \R^3: x=y=0\},
	\end{equation*}
	if $a = -d$ and 
	\begin{equation*}
		Fix(\varphi)= \{(0,0,0)\}, 
	\end{equation*}
	if $a \neq -d.$ 
	
	\noindent\textbf{Case 1: $A$ is not invertible:} The determinant of $A$ give us the expression 
	\begin{equation*}
		ad = cb. 
	\end{equation*}
	
	WLOG, let us suppose that $a \neq 0$. Then $d = \frac{cb}{a}$. From system (\ref{linearsystem0}), we get $x = \frac{-by}{a}$. In the equation of $\dot{z}$, we obtain 
	\begin{equation}\label{eqyzz}
		\left(f - \frac{be}{a}\right)y + (a+d)z = 0. 
	\end{equation}
	
	Spliting into cases, if $a\neq -d$ and $f \neq \frac{be}{a}$, we get 
	\begin{equation*}
		Fix(\varphi) = \left\{(x,y,z) \in \mathcal{H}: x = \frac{-by}{a}, \left(\frac{be}{a}- f\right)y = (a+d)z\right\}. 
	\end{equation*}
	
	If $a  = -d$ and $f \neq \frac{be}{a}$ then 
	\begin{equation*}
		Fix(\varphi) = \left\{(x,y,z) \in \mathcal{H}: x = y = 0\right\}. 
	\end{equation*}
	
	If $a \neq -d$ and $f = \frac{be}{a}$, then 
	\begin{equation*}
		Fix(\varphi) =\left\{(x,y,z) \in \mathcal{H}: z = 0, x = \frac{-by}{a}\right\}. 
	\end{equation*}
	
	If $a = -d$ and $f = \frac{be}{a}$, then 
	\begin{equation*}
		Fix(\varphi) =\left\{(x,y,z) \in \mathcal{H}: x = \frac{-by}{a}\right\}. 
	\end{equation*}
	
	\subsection{Subgroups of $\mathcal{H}$ and observability}
	
A key challenge is to classify the subgroups of $\mathcal{H}$ to examine all homomorphisms between the Heisenberg group and its subgroups. Based on the classification provided in [5, Proposition 2], and focusing exclusively on closed and simply connected subgroups, up to isomorphisms we have only the following sets:	
 
	\begin{equation*}
		\begin{array}{cl}
			H_1 = &\{(x,y,z) \in \mathcal{H}: z = 0\}  \\
			H_2 = &\{(x,y,z) \in \mathcal{H}:  y = 0\}  \\
			H_3 = &\{(x,y,z) \in \mathcal{H}: x = z = 0\}  \\
			H_4 = &\{(x,y,z) \in \mathcal{H}: x = y = 0\}  \\
			H_5 = &\{(x,y,z) \in \mathcal{H}: y = z = 0\}  \\
			H_6 = &\{(x,y,z) \in \mathcal{H}: x = y, z = 0\}  \\
			H_7 = &\{(x,y,z) \in \mathcal{H}: x = z, y = 0\}. \\
		\end{array}
	\end{equation*}
	Furthermore, given $\hat{a},\hat{b} \in \R$ with $\hat{a}\hat{b} \neq 0$, we obtain
	\begin{eqnarray*}
		H_8 = &\{(\hat{a}t,\hat{b}t,0) \in \mathcal{H}: t \in \R\} \\
		H_9 = &\{(\hat{a}s,0,\hat{b}s) \in \mathcal{H}: s \in \R\}.
	\end{eqnarray*}
	
	To illustrate the technique we will use in this section, let us work out an example. 
 
Consider the Lie subgroup $H_1$ above  and the map $h: \mathcal{H} \longrightarrow H_1$ given by \begin{equation}\label{h1}
			h(x,y,z) = (y,z,0). 
		\end{equation}
		
		We claim  that $h$ is a Lie homomorphism. In fact, 
		\begin{eqnarray*}
			h((x_1,x_2,x_3) \cdot (y_1,y_2,y_3)) &=& h(x_1 + y_1 + x_2y_3,x_2 + y_2, x_3 + y_3)\\
			&=&(x_2 + y_2, x_3 + y_3,0).
		\end{eqnarray*}
		And, on the other hand 
		\begin{eqnarray*}
			h(x_1,x_2,x_3) \cdot h(y_1,y_2,y_3) &=& (x_2,x_3,0)\cdot(y_2,y_3,0)\\
			&=&(x_2 + y_2, x_3 + y_3,0).
		\end{eqnarray*}
		
	Therefore, we obtain $$h((x_1,x_2,x_3) \cdot (y_1,y_2,y_3)) = h(x_1,x_2,x_3) \cdot h(y_1,y_2,y_3).$$ It is clear that $h$ is differentiable and thus a Lie homomorphism. 
  
  Let us denote by $K = \ker(h)$, the kernel of the out map. It turns out that
		\begin{equation}\label{K_1}
			K = \{(x,y,z) \in \mathcal{H}: y = z = 0\}. 
		\end{equation}
		
		Next, we consider the general pair $(\mathcal{X},\pi_K)$, where $\pi_K: \mathcal{H} \longrightarrow \mathcal{H}/K$ is the canonical projection. 
  
  Regarding to the homomorphism $h$ in (\ref{h1}), we get the following proposition.
		
\begin{proposition}Consider the linear pair $(\mathcal{X}, \pi_K)$, with $K$ given by (\ref{K_1}). Then, $(\mathcal{X},\pi_K)$ is observable if $a+d\neq 0$ and $e \neq 0$.  \end{proposition}
		
\begin{proof}Denotes by $\varphi_t(x,y,z) = (x(t),y(t),z(t))$ the solution of $\mathcal{X}$. Since $\varphi_t(x,y,z) \in K$, the formula in (\ref{Sylvexp}) implies that $cs_1(t)x + (ds_1(t) + s_0(t))y = 0$, for all $t \in \R$. Considering $t=0$, we get $y = 0$. 

Through the system (\ref{odeX}) we get $cx = 0, $ and the solution is in the form
			\begin{equation}\label{systemy=0}
				\varphi_t(x,0,z) = \left(e^{ta}x, 0, e^{(a+d)t}\left(z + \frac{ex}{a+d}\right) - \frac{e e^{ta}x}{a+d}\right). 
			\end{equation}
			
			Since $e^{(a+d)t}\left(z + \frac{ex}{a+d}\right) - \frac{e e^{ta}x}{a+d} = 0$, with $t=0$ we get $z = 0$. Therefore, 
			\begin{equation*}
				e^{(a+d)t}\left(\frac{ex}{a+d}\right) - \frac{e e^{ta}x}{a+d} = 0,
			\end{equation*}
			and we get $x = y = z = 0$. Thus, $I = \{(0,0,0)\}$.
   
   Consequently, by \cite[Corollary 2.6]{Ayalaetal}, the pair $(\mathcal{X},\pi_K)$ is locally observable. 
			
			Now, regarding to the set $Fix(\varphi)$, if $y = z = 0$, we get for any of the cases in Section (\ref{fixedpoints}) that $$Fix(\varphi) \cap K = \{(0,0,0)\}.$$ Therefore, by \cite[Theorem 2.5]{AyalaHaci}, the system is observable. 
		\end{proof}
	\end{remark}

In the following subsections, we will examine the general form of the homomorphism $h$ and identify the hypotheses required to ensure observability.
	
	\subsubsection{Subgroups $H_1$ and $H_2$}
	
	In order to consider a general case, let us suppose that $h: \mathcal{H} \longrightarrow H_1$ is given by 
	\begin{equation}\label{h2}
		h(x,y,z) = (h_1(x,y,z), h_2(x,y,z),0). 
	\end{equation}
	where $h_i: \mathcal{H}  \longrightarrow \R$, according to the product (\ref{productH}), must satisfy 
	\begin{equation}\label{hgeral}
		h_i (x_1 + y_1 + x_2y_3,x_2 + y_2,x_3 + y_3) = h_i(x_1,x_2,x_3) + h_i(y_1,y_2,y_3), i = 1,2.  
	\end{equation}
	
	If $h_i(x,y,z) = \alpha_i x + \beta_i y + \gamma_i z$, then 
	\begin{equation*}
		\alpha_i(x_1 + y_1 + x_2y_3) + \beta_i(x_2 + y_2) + \gamma_i(x_3 + y_3) = \alpha_i x_1 + \gamma_i x_2 + \gamma_i x_3 + \alpha_i y_1 + \beta_i y_2 + \gamma_i y_3.  
	\end{equation*}
	
	It is easy to check, by the equality above that $\alpha_i = 0$, for $i=1,2.$ Therefore 
	\begin{equation*}
		h(x,y,z) = (\beta_1 y + \gamma_1 z, \beta_2 y + \gamma_2 z,0). 
	\end{equation*}
	
	Consequently, the kernel of $h$ is the solution of the linear system 
	\begin{equation}\label{kernellinearsystem}
		\left\{\begin{array}{cc}
			\beta_1 y + \gamma_1 z &= 0  \\
			\beta_2 y + \gamma_2 z &= 0 
		\end{array}\right.
	\end{equation}
	
	Let us analyze the conditions for the observability of the linear pair $(\mathcal{X}, \pi_K)$. Considering the expression for $h$, given in (\ref{h1}), for the general case where $h$ is defined in (\ref{h2}), we face the linear system in (\ref{kernellinearsystem}). 
 
 By introducing the matrix $B = \begin{bmatrix}
		\beta_1 & \gamma_1\\
		\beta_2 & \gamma_2
	\end{bmatrix}$, we get the following result. 
	
	\begin{proposition}If $B$ is invertible, $a+d \neq 0$ and $e \neq 0$, the linear pair $(\mathcal{X}, \pi_K)$ is observable. 
	\end{proposition}
	
	\begin{proof}The hypothesis above guarantees that the set $K$ is given by 
		\begin{equation*}
			K = \{(x,y,z) \in \mathcal{H}: y = z = 0\}. 
		\end{equation*}
		
		Using again the expression (\ref{odeX}), we get $\varphi_t(x,y,z) \in K$ if $y = z = 0$. Consequently, $cx = ex = 0$. Through the ordinary differential equation 
		\begin{equation*}\left\{
			\begin{array}{ccc}
				\dot{X} &=& \mathcal{X}(x,y,z),\\
				X(0)&=& (x_0,y_0,z_0), 
			\end{array}\right.
		\end{equation*}
		we get the same solution as in (\ref{systemy=0}), which implies local observability by \cite[Corollary 2.6]{Ayalaetal}.
		
		The set $Fix(\varphi)$ was given in the previous section. In any case, $A$ is invertible or not, we get $$Fix(\varphi) \cap K = \{(0,0,0)\},$$ which implies observability.   
	\end{proof}
	
	The case when $\det B = 0$ is more unstable. If $B$ is not invertible, we define the homomorphism
	\begin{equation*}
		h(x,y,z) = \left(\beta_1 y + \gamma_1 z, \beta_2 y + \frac{\beta_2 \gamma_1}{\beta_1} z,0\right).
	\end{equation*}
	A straightforward computation shows that
	\begin{equation*}
		\ker h = \{(x,y,z) \in \mathcal{H}: \beta_1 y = -\gamma_1 z\}
	\end{equation*}
	
	Considering any possible case, we get 
	\begin{equation}\label{kernvariouscases}
		\ker h =
		\left
		\{\begin{array}{cl}
			\mathcal{H},& \beta_1 = \gamma_1 = 0,\\
			\{(x,y,z) \in \mathcal{H}: y = 0\},& \beta_1 \neq 0, \gamma_1 = 0,\\
			\{(x,y,z) \in \mathcal{H}: z = 0\},& \beta_1 = 0, \gamma_1 \neq 0,\\
			\{(x,y,z) \in \mathcal{H}: y = \frac{-\gamma_1}{\beta_1}z\},& \beta_1 \neq 0, \gamma_1 \neq 0. 
		\end{array}\right.
	\end{equation}
	the discussion about observability is in the next remark.  
	
	\begin{remark}\label{remarkesq}Let us discuss some specific conditions for the linear pair $(\mathcal{X}, \pi_K)$ when the kernel is given by the expression (\ref{kernvariouscases}). At first, the case $\gamma_1 = \beta_1 = 0$ implies obviously that $(\mathcal{X}, \pi_K)$ is not observable. The case $\beta_1 \neq 0$ and $\gamma_1 = 0$ generates the ODE system 
		\begin{equation}\label{odeXy=00}
			\left\{
			\begin{array}{ccl}
				\dot{x} &=& ax\\
				\dot{z} &=& ex + (a+d)z 
			\end{array}\right.
		\end{equation}
		which also get us to the solution in (\ref{systemy=0}) with $z(t)$ not necessarily zero, implying in $I = K$ and consequently, the pair $(\mathcal{X}, \pi_K)$ is not locally observable.  
		
		Now, considering the solution in (\ref{generalsolutionheis}), for the case when $\beta_1 = 0$ and $\gamma_1 \neq 0$, if $\varphi_t(x,y,z) \in K$ for all $t \in \R$, we get that $z = 0$ and
		\begin{equation*}
			\left\langle\int_0^te^{s(A-(a+d)I_2)^T}(e,f)ds, (x,y)\right\rangle = 0
		\end{equation*}
		
		Considering 
		\begin{equation*}
			\int_0^te^{s(A-(a+d)I_2)^T}(e,f)ds = (f_1(t),f_2(t)), 
		\end{equation*}
		we obtain 
		\begin{equation*}
			\left\langle\int_0^te^{s(A-(a+d)I_2)^T}(e,f)ds, (x,y)\right\rangle = f_1(t)x + f_2(t)y = 0,
		\end{equation*}
		which satisfies $f_1(0) = f_2(0) = 0$. 
  
  Consider $I_j = \{t \in \R: f_i(t) \neq 0\}$ for $j = 1,2$ and $J = I_1 \cap I_2$. For the case when $J = \emptyset,$ the equality $f_1(t)x + f_2(t)y=0$ implies that $f_1(t)x = 0$ with $f_1(t) \neq 0$ when $t\in I_1$ and $f_2(t)y=0$ with $f_2(t) \neq 0$ when $t \in I_2$. Therefore $x = y = 0$. We conclude that $I= \{(0,0,0)\}$ and the pair $(\mathcal{X}, \pi_K)$ is locally observable. 
  
  Now, if $J \neq \emptyset,$ the equality $f_1(t)x + f_2(t)y = 0$ generates a line in $\R^2$ for each $t \in \R$, which shows that $I$ is can not be discrete. Therefore, the linear pair $(\mathcal{X}, \pi_K)$ is not observable. 
	\end{remark}

	As a conclusion, the only general case for the pair $(\mathcal{X}, \pi_K)$ to be observable is when $c \neq 0$ and $B$ invertible.  
	
	\begin{remark}For the subgroup $H_2$, consider the case when $h: \mathcal{H} \longrightarrow H_2$ is given by 
		\begin{equation*}
			h(x,y,z) = (\alpha_1 x + \beta_1 y + \gamma_1 z, 0 , \alpha_2 x + \beta_2 y + \gamma_2 z). 
		\end{equation*}
		
		As the same reasoning in (\ref{hgeral}), we obtain $\alpha_1 = \alpha_2 = 0$. Therefore 
		\begin{equation*}
			h(x,y,z) = (\beta_1 y + \gamma_1 z, 0 ,\beta_2 y + \gamma_2 z)
		\end{equation*}
		
		One can notice that the kernel $K = \ker(h)$ can be also found using the system (\ref{kernellinearsystem}) and the same cases pre-estabilished. Therefore, the linear pair $(\mathcal{X}, \pi_K)$ is locally observable if $B = \begin{bmatrix}
			\beta_1 & \gamma_1 \\
			\beta_2 & \gamma_2
		\end{bmatrix}$ is invertible and $c \neq 0$. Also, under the same hypothesis we get $Fix(\varphi) \cap K = \{(0,0,0)\}$ and, consequently, the pair $(\mathcal{X}, \pi_K)$ is observable. 
	\end{remark}
	
	\subsubsection{Case $H_3$, $H_4$ and $H_5$}\label{H_3}
	
	The homomorphisms between $\mathcal{H}$ and $H_3$ are given by the maps $h: \mathcal{H} \longrightarrow H_3$ defined by 
	\begin{equation}\label{kerneldoisdim}
		h(x,y,z) = (0,\beta y + \gamma z, 0), 
	\end{equation}
	for some constants $\beta, \gamma \in \R$. For this case, we have 
	\begin{equation}\label{kernel1dim}
		\ker h = \{(x,y,z) \in \mathcal{H}: \beta y = -\gamma z\}. 
	\end{equation}
	which lead us to the same case in (\ref{kernvariouscases}), whose explanation is in the Remark (\ref{remarkesq}). 
	
	Now, for the subgroup $H_4$, considering $h: \mathcal{H} \longrightarrow H_4$ in the form $h(x,y,z) = (0,0,\beta y + \gamma z), $ we get that 
	\begin{equation*}
		K = \{(x,y,z) \in \mathcal{H}: \beta y = -\gamma z\}. 
	\end{equation*}
	which is precisely the previous case. The same reasoning can be made for the subgroup $H_5$. Now, for the subgroup $H_6$, considering the map $h: \mathcal{H} \longrightarrow H_6$ given by $h(x,y,z) = (\beta y + \gamma z, 0,\beta y + \gamma z)$, we get that 
	\begin{equation*}
		K = \{(x,y,z) \in \mathcal{H}: \beta y = -\gamma z\}. 
	\end{equation*}
	which is also the case for $H_3$. The case for $H_7$ is similar.  
	
	\subsubsection{Cases $H_8$ and $H_9$}
	
	Firstly, let us find the homomorphisms between $\mathcal{H}$ and $H_8$. By the same reasoning shown previously, we can take $h: \mathcal{H} \longrightarrow H_8$ by 
	\begin{equation*}
		h(x,y,z) = (\beta_1 y + \gamma_1 z , \beta_2 y + \gamma_2 z, 0). 
	\end{equation*}
	
	One can notice that the map above must satisfy 
	\begin{equation*}
		\left\{
		\begin{array}{cc}
			\beta_1 y + \gamma_1 z = t \hat{a}\\
			\beta_2 y + \gamma_2 z = t \hat{b}
		\end{array}\right.
	\end{equation*}
	for some $t \in \R$. This implies in 
	\begin{equation*}
		t = \frac{\beta_1 y + \gamma_1 z}{\hat{a}} = \frac{\beta_2 y + \gamma_2 z}{\hat{b}}. 
	\end{equation*}
	that is, 
	\begin{equation*}
		(\beta_1 y + \gamma_1 z)\hat{b} = (\beta_2 y + \gamma_2 z)\hat{a}, \forall y,z \in \R.  
	\end{equation*}
	
	Taking $y = 1$ and $z = 0$ we get $\beta_1 \hat{b} = \beta_2 \hat{a}$. Also, if $y = 0$ and $z = 1$, we get $\gamma_1 \hat{b} = \gamma_2 \hat{a}$. Therefore, considering $\alpha = \frac{\beta_1}{\hat{a}} = \frac{\beta_2}{\hat{b}}$ and $\beta = \frac{\gamma_1}{\hat{a}} = \frac{\gamma_2}{\hat{b}}$, we obtain 
	\begin{equation}\label{kernelline}
		h(x,y,z) = ((\alpha y + \beta z)\hat{a}, (\alpha y + \beta z)\hat{b}, 0). 
	\end{equation}
	the final form of the homomorphism $h$. Let us find the kernel of $K$. As a matter of fact, we need to study the equation 
	\begin{equation*}
		(\alpha y + \beta z) \hat{a} = 0. 
	\end{equation*}
	
	As we are considering $\hat{a} \neq 0$, we get $\alpha y + \beta z = 0$, which is the same equation in the expression for the kernel in (\ref{kernel1dim}), whose conditions for observability are discussed in the Remark (\ref{remarkesq}). For the subgroup $H_9$, considering the same reasoning above, we would get a homomorphism $h: H \longrightarrow H_9$ such that 
	\begin{equation*}
		h(x,y,z) = ((\alpha y + \beta z)\hat{a}, 0 , (\alpha y + \beta z) \hat{b}). 
	\end{equation*}
	which generates the same kernel as the previous case.
	
	\begin{remark}The case when $h \equiv 0$, that is, the homomorphism $h$ is trivial, we get that $K = \mathcal{H}$. Therefore, the system can not be observable, given that $\varphi_t(x,y,z) \in K,$ for every $(x,y,z) \in \mathcal{H}$, implying the set $I$ not discrete at all.
	\end{remark}
	
\section*{Conclusion and future research}
	
In this article, we introduced a precise method for determining the observability property of general pairs within the Heisenberg three-dimensional Lie group. This technique utilizes homomorphisms between the group and its simply connected subgroups, resulting in a quotient space as the output. Under certain assumptions about the system and the subgroup in question, we established several results that can be valuable for analyzing the system's dynamics from the perspective of the output space.

To avoid impractical scenarios, we identified subgroups in which the system cannot be observed, as well as some unusual cases where the system exhibits instabilities, as discussed in Remark (\ref{remarkesq}). This analysis uses results referenced in Section 2. 

For future research, we strongly believe that the techniques employed here could also be valuable for examining observability in non-nilpotent solvable affine three-dimensional Lie groups, which will be our next focus.
 	
\section{Acknowledgments:} The author A. J. Santana is partially supported by CNPq grant n. 309409/2023-3

\section{Conflict of interest:} The authors state no conflicts of interest.


\begin{thebibliography}{99}


		\bibitem{Ayalaetal} V. Ayala, A. Hacibekiroglu, E. Kizil, \textit{Observability of general linear pairs}, Computer $\&$ mathematics with applications. Elsevier, vol 39, (2000), 35-43.  
		
		\bibitem{AyalaandAdriano} V. Ayala, A. Da Silva,  \textit{On the characterization of the controllability property
			for linear control systems on nonnilpotent, solvable three-dimensional Lie groups}, Journal of Differential Equations, vol. 266, (2019),  8233-8257. 
		
		\bibitem{AyalaHaci}  V. Ayala, A. Hacibekiroglu, \textit{Observable linear pairs}, Computation and Applied Mathematics,  v. 13, (1997),  205-214. 

\bibitem{AyTi} V.~Ayala and J.~Tirao, \emph{Linear
control systems on Lie groups and Controllability,} Proceedings of Symposia in Pure Mathematics, Vol 64, AMS, 1999, 47-64. 
  
		
		\bibitem{Berger} T. Berger, T. Reis, S. Trenn,  \textit{Observability of linear differential-algebraic systems: A survey},  Surveys in Differential-Algebraic Equations IV, Differential-Algebraic Equations Forum, Springer, 2017.
		
		
		\bibitem{adrianoheis} A. Da Silva, E. Kizil,O. Duman, \textit{Linear control systems on homogeneous spaces of the Heisenberg group}, J Dyn Control Syst, (2024),  2065–2086.
		
		\bibitem{DathAndJouan} M. Dath, P. Jouan, \textit{Controllability of linear systems on low dimensional
			nilpotent and solvable Lie groups}, J Dyn Control Syst, 22, (2016), 207–225.
		
		\bibitem{Etlili}D. Etlili, O. Naifar, A. Errachdi,   \textit{Observers with Unknown Inputs of Linear Systems},  Advances in Observer Design and Observation for Nonlinear Systems. Studies in Systems, Decision and Control, vol 410, Springer, 2022.
		
		\bibitem{Kalman1} R.E. Kalman, \textit{On the general theory of control systems}, IFAC Proceedings, Vol. 1, (1960), 491–502. 
		
		\bibitem{Kalman2} R.E. Kalman,  \textit{Mathematical Description of Linear Dynamical Systems}, J.S.I.A.M. Control. vol 1, (2), (1963), 152–192. 
		
		\bibitem{perko} L. Perko,  \textit{Differential Equations and Dynamical Systems}, Springer. 2001.
		
		\bibitem{Sivalingam} S.M. Sivalingam, V. Govindaraj, \textit{Observability of Time-Varying Fractional Dynamical Systems with Caputo Fractional Derivative}, Mediterr. J. Math. 21, (2024),  75-96.

\bibitem{Pontryaguin}
L.S.~Pontryagin, V.G.~Boltyanskii, R.V.~Gamkrelidze, E.F.~Mishchenko, \emph{The
mathematical theory of optimal processes}, John Wiley and Sons, New York-London 1962. 
  
		\bibitem{Sontag} E. Sontag,  \textit{ Mathematical Control Theory}, Springer New York, NY. 1998. 
		
		\bibitem{onish} A.L. Onishchik, E.B. Vinberg, \textit{Lie groups and Lie algebras}, Springer-Verlag, 1993. 
		
		
		
	\end{thebibliography}
\end{document}